\documentclass[10pt]{amsart}
\usepackage[cp1251]{inputenc}
\usepackage[english,russian]{babel}
\usepackage{amsmath}
\usepackage{amssymb}
\usepackage{amsfonts}
\usepackage{graphicx}

\newtheorem{theorem}{Theorem}

\newtheorem{corollary}{Corollary}

\newcommand{\E}{{\mathbf E}}

\begin{document}
\renewcommand{\refname}{References}
\renewcommand{\proofname}{Proof.}
\renewcommand{\figurename}{Fig.}

\thispagestyle{empty}

\title[Limit theorems in infinite urn schemes]{Limit theorems for forward and backward processes 
of numbers of non-empty urns
in infinite urn schemes}
\author{{M.G. Chebunin, A.P. Kovalevskii}}%
\address{Mikhail Georgievich Chebunin
\newline\hphantom{iii} Karlsruhe Institute of Technology,
\newline\hphantom{iii} Institute of Stochastics,
\newline\hphantom{iii} 76131, Karlsruhe, Germany;
\newline\hphantom{iii} Novosibirsk State University, 
\newline\hphantom{iii} Pirogova str., 1,
\newline\hphantom{iii} 630090, Novosibirsk, Russia}%

\email{chebuninmikhail@gmail.com}%

\address{Artyom Pavlovich Kovalevskii
\newline\hphantom{iii} Novosibirsk State Technical University, 
\newline\hphantom{iii} K. Marksa ave., 20,
\newline\hphantom{iii} 630073, Novosibirsk, Russia%
\newline\hphantom{iii} Novosibirsk State University, 
\newline\hphantom{iii} Pirogova str., 1,
\newline\hphantom{iii} 630090, Novosibirsk, Russia}%

\email{artyom.kovalevskii@gmail.com}%

\thanks{\sc Chebunin, M.G., Kovalevskii, A.P.,
Limit theorems for forward and backward processes 
of numbers of non-empty urns
in infinite urn schemes}
\thanks{\copyright \ 2022 Chebunin M.G., Kovalevskii A.P.}
\thanks{\rm The work is supported by Mathematical Center in Akademgorodok under agreement No. 075-15-2019-1675 with the Ministry of Science and Higher Education of the Russian Federation.}
\thanks{\it Received November, 1, 2022, published December, 1,  2022.}%

\maketitle {\small
\begin{quote}
\noindent{\sc Abstract. } 

%

We study the joint asymptotics of  forward and backward processes of  numbers of non-empty urns in an infinite urn scheme.
The probabilities of balls hitting the urns are assumed to satisfy the conditions of regular decrease.
We prove weak convergence to a two-dimensional Gaussian process. Its covariance function depends only on
exponent of regular decrease of probabilities. We obtain parameter estimates that have a normal asymototics for its joint distribution together with forward and backward processes.
We use these estimates to construct statistical tests for the homogeneity of the urn scheme on the number of thrown balls.
\medskip

 \end{quote}
}

\section{Introduction}

Let $X_1,\ldots,X_n$ be independent and identically distributed positive integer-valued random variables,
\begin{equation}\label{a1}
p_i={\bf P}(X_1=i), \quad \sum_{i=1}^{\infty}p_i=1.
\end{equation}

The number of different elements among first $k$ ones ($2\le k\le n$) is
\begin{equation}\label{a2}
R_k=1+\sum_{i=2}^k {\bf 1}(X_i\not \in \{X_1,\ldots,X_{i-1}\}).
\end{equation}

Similarly, the number of different elements among {\em last} $k$ ones ($2\le k\le n$) is
\begin{equation}\label{a3}
R_k'=1+\sum_{i=n-k+1}^{n-1} {\bf 1}(X_i\not \in \{X_{n-k+2},\ldots,X_{n}\}).
\end{equation}

In other words, 
\begin{equation}\label{a4}
R_k'=1+\sum_{i=2}^k {\bf 1}(X_i'\not \in \{X_1',\ldots,X_{i-1}'\}),
\end{equation}
with $X_i'=X_{n-i+1}$, the random variables in the 
backward order, $1 \le i \le n$.

We put by definition
\begin{equation}\label{a5}
R_0=R_0'=0, \quad R_1=R_1'=1.
\end{equation}

Distributions of $R_n$ and $R_n'$ are identical, their limiting properties are known.
We study their limiting joint disrtibution under the appropriate centering and normalizing.

If there is an infinite number of positive probabilities in (\ref{a1}) then this probability model is the infinite
urn scheme. 
Karlin (1967) established the SLLN for $R_n$ in the infinite urn scheme (Bahadur (1960) proved the weak LLN),
\begin{equation}\label{a6}
R_n/{\bf E} R_n \to 1 \quad \text{a.s.}
\end{equation}

Now we need the regularity condition. Let $p_1 \ge p_2\ge \ldots >0$ and
\begin{equation}\label{a7}
\alpha(x):=\max\{k>0: \ p_k \ge 1/x\}= x^{\theta} L(x) \quad \text{as} \quad x \to \infty, \quad 
0<\theta< 1,
\end{equation}
$L(\cdot)$ is the slowly varying function of the real argument: 
$L(tx)/L(x)\to 1$ as $x \to  +\infty$ for any real $t>0$.  

From Karamata's characterization theorem, $\alpha(x)$ is a regulary varying function with index $\theta$.
The model (\ref{a6}) is the elementary probability model that corresponds to the Zipf's Law (Zipf, 1936) 
of power decreasing of  word probabilities.

Karlin (1967) proved the CLT: if (\ref{a6}) holds then $(R_n - {\bf E} R_n)/\sqrt{{\bf Var} R_n}$ converges weakly 
to the standard normal distribution,
\begin{equation}\label{a8}
{\bf E} R_n \sim \Gamma(1-\theta)\alpha(n), \quad {\bf Var} R_n/{\bf E} R_n \to 2^{\theta}-1,
\end{equation}
$\Gamma(\cdot)$ is the Euler gamma.

From the Karlin's CLT and (\ref{a7}), $(R_n - {\bf E} R_n)/\sqrt{{\bf E} R_n}$ converges weakly 
to the centered normal distribution with variance $2^{\theta}-1$. The CLT holds for $\theta=1$ too
but with another normalization.

Chebunin and Kovalevskii (2016) proved the Functional CLT: if (\ref{a6}) holds then the process 
\begin{equation}\label{a9}
Z_n=\{Z_n(t), \ 0 \le t \le 1\}=\{(R_{[nt]}-{\bf E} R_{[nt]})/\sqrt{{\bf E} R_n}, \ 0\le t \le 1\}
\end{equation}
 converges weakly in $D(0,1)$ with uniform metrics to a centered Gaussian
process $ Z_{\theta} $ with continuous a.s. sample paths and covariance function 
\begin{equation}\label{a10}
K(s,t)=(s+t)^{\theta}-\max(s^{\theta}, t^{\theta}).
\end{equation}

The Karlin's CLT is a partial case of the FCLT for $Z_n(1)$. The same FCLT is true for
\begin{equation}\label{a11}
Z_n'=\{Z_n'(t), \ 0 \le t \le 1\}=\{(R_{[nt]}'-{\bf E} R_{[nt]})/\sqrt{{\bf E} R_n}, \ 0\le t \le 1\}.
\end{equation}

We prove the theorem about the joint limiting distribution of $(Z_n, Z_n')$.

All the papers on properties of $R_n$ and similar statistics in the infinite urn scheme can be divided into 4 types:

1. Results under the regularity condition (\ref{a7}): the papers above, Durieu $\&$ Wang (2016),
Chebunin (2017),  Durieu, Samorodnitsky $\&$ Wang (2020), Chebunin $\&$ Zuyev (2020).

2. Results under (\ref{a7}) with $\theta=0$ instead of $0<\theta\le 1$, that is, for the slowly varying function $\alpha(x)$:
Dutko (1989), Barbour (2009), Barbour $\&$ Gnedin (2009).

3. Results for the model without assuming the regularity condition (\ref{a7}): Key (1992, 1996), 
Hwang  $\&$ Janson (2008), Muratov $\&$ Zuyev (2016),
Ben-Hamou, Boucheron $\&$ Ohannessian (2017), Decrouez, Grabchak $\&$  Paris (2018).

4. Statistical applications --- we postpone the survey of these results to Section~2. 

Gnedin, Hansen $\&$ Pitman (2007) made a detalied survey of the results of types 1--3 existed at the time.

\section{Main results}

\begin{theorem}

If (\ref{a7}) holds then the process
$
(Z_n, Z'_n)=
\left\{(Z_n(t), Z'_n(t)),
\ 0 \leq t \leq 1 \right\}
$
converges weakly in the uniform metrics in  $D(0,1)^2$ 
to 2-dimensional Gaussian process $(Z, Z')$ with zero expectation and covariance function 
\[
{\bf E} Z(s)Z(t)={\bf E} Z'(s)Z'(t)=K(s,t), \ \ {\bf E} Z(s)Z'(t)=K'(s,t), 
\]
where $K(s,t)$ done by (\ref{a10}),  and 
\begin{equation}\label{a11}
K'(s,t)= ((s+t)^\theta-1){\bf  1}(s+t>1).
\end{equation}
\end{theorem}

From Theorem 1 we have that the limiting process $\{(Z(t)-Z'(t))/\sqrt{2}, \ 0 \le t \le 1/2\}$ is the 
stochastically self-similar process which coinside in distribution with the limiting process of  Durieu and Wang (2016).
So Theorem 1 gives an alternative way to simulate these processes without additional randomization.

We need some estimate of the unknown parameter $\theta$ to use the theorem in applications.
 Various  classes of such estimates  have been obtained and analysed  by Hill (1975), Nicholls (1978), 
Zakrevskaya and Kovalevskii (2001, 2019), Guillou and Hall (2002), Ohannessian and Dahleh (2012), 
Chebunin (2014), Chebunin and Kovalevskii (2019a, 2019b), Chakrabarty et al. (2020).

But we need an estimate that is symmetric to the forward and backward processes.
Moreover, we want to have the limiting joint distribution of the estimate and the two-dimensional process.
We introduce the estimate and study its propetries in the next section.

\section{Parameter's estimation }

From (\ref{a6}) and (\ref{a8}),  we have $\log R_n\sim{\theta}\log n$  a.s.
Therefore, we may propose the following estimators for parameter $\theta$:
\[
\theta_n=\int_0^1  \log^+ R_{[nt]} \, d A(t), \ \ \ \theta'_n=\int_0^1  \log^+ R'_{[nt]} \, d A(t),
\]
here $\log^+ x=\max(\log x, \, 0)$.
Function $A(\cdot)$ has bounded variation and
\begin{equation}\label{intalog}
A(0)=A(1) =0, \ \ \lim_{x \downarrow 0} \log x \int_0^x |dA(t)| =0, \ \ \int_0^1 \log t \, dA(t) =1. 
\end{equation}
Let
\[
\widehat{\theta}=(\theta_n+\theta'_n)/2.
\]

\begin{theorem}
Let $p_i =i^{-1/\theta} l(i,\theta)$, $\theta\in [0,1]$, and  $l(x,\theta)$ 
is a slowly varying function as $x \to \infty$. Then 
the estimator $\widehat{\theta}$ is strongly consistent.
\end{theorem}
\begin{proof}
The proof follows from the definition of $\widehat{\theta}$ and Theorem 1 from Chebunin and Kovalevskii (2019).
\end{proof}

We need extra conditions to obtain the asymptotic normality of $\widehat{\theta}$.
 
\begin{theorem}
Let $p_i=ci^{-1/\theta}(1+o(i^{-1/2}))$,
 $\theta \in (0,1)$, and
 $A(t)=0$, $t \in [0,\, \delta]$ for some $\delta\in(0,\, 1)$. Then
\[
\sqrt{{\bf E}R_n} (\widehat{\theta}-\theta) - \frac12 \int_0^1  t^{-\theta} (Z_n(t)+Z'_n(t)) \, dA(t) \to_p 0.
\]
\end{theorem}
\begin{proof}
The proof follows from the definition of $\widehat{\theta}$ and Theorem 2 from Chebunin and Kovalevskii (2019).
\end{proof}

From Theorem 3, it follows that  $\widehat{\theta}$ converges to $\theta$ with rate $({\bf E}R_n)^{-1/2}$, and 
$\sqrt{{\bf E}R_n} (\widehat{\theta}-\theta)$ converges weakly to the normal random variable 
$ \frac12 \int_0^1  t^{-\theta}( Z_{\theta}(t)+ Z'_{\theta}(t)) \, dA(t)$ with variance
$ \frac12 \int_0^1 \int_0^1 (st)^{-\theta} (K(s,t)+k(s,t)) \, dA(s) \, dA(t)$.

{\bf Example 1} Take
\[
A(t) = \left\{
\begin{array}{ll}
0, & 0 \le t \le 1/2; \\
-(\log 2)^{-1}, & 1/2<t<1; \\
0, & t=1.
\end{array}
\right.
\]

Then
\[
\widehat{\theta} = \log_2 \left(R_n/\sqrt{R_{[n/2]} R'_{[n/2]}}\right), \ \ n  \ge 2.
\]

\section{Test for a known rate}

Let $0<\theta<1$ be known. We introduce  {\em empirical bridges} $\overset{o}{Z}_n$,  $\overset{o}{Z'}_n$  (Kovalevskii and Shatalin, 2015, 2016) as follows.
\[
\overset{o}{Z}_n(k/n)=\left(R_k-(k/n)^{\theta}R_n\right)/\sqrt{R_n}, \ \ \overset{o}{Z'}_n(k/n)=\left(R'_k-(k/n)^{\theta}R_n\right)/\sqrt{R_n},
\]
$0\leq k \leq n$, where $R_0=0$. We construct a piecewise linear approximation: 
for any $0\leq u < 1/n$ and $0\leq k \leq n-1$,
\[
\overset{o}{Z}_n\left(\frac{k}{n}+u\right)=\overset{o}{Z}_n(k/n)+nu\left(\overset{o}{Z}_n((k+1)/n) - \overset{o}{Z}_n(k/n)\right),
\] 
\[
\overset{o}{Z'}_n\left(\frac{k}{n}+u\right)=\overset{o}{Z'}_n(k/n)+nu\left(\overset{o}{Z'}_n((k+1)/n) - \overset{o}{Z'}_n(k/n)\right).
\] 

\begin{theorem}
Under the assumptions of Theorem 2,
\[
\sup_{0\le t \le 1} |\overset{o}{Z}_n(t)-(Z_n(t)-t^{\theta}Z_n(1))| \to 0 \ \mbox{a.s.}
\]
\[
\sup_{0\le t \le 1} |\overset{o}{Z'}_n(t)-(Z'_n(t)-t^{\theta}Z'_n(1))| \to 0 \ \mbox{a.s.}
\]
\end{theorem}
\begin{proof}

The first statement is Theorem 3 from the Ch.K(2019). 
For the second statement let $t\in [0,1)$, and $k=[nt]$, then $t=k/n+u$, $0 \le k \le n-1$, $u\in[0,1/n)$.
Let $f_{\theta}(x)=(1+x)^{\theta}-x^{\theta}$. So $0\le f_{\theta}(x)\le f_{\theta}(0)=1$ for $x\ge 0$.

By the definition of $\overset{o}{Z'}_n(t)$,
\[
\frac{R'_k - \left(\frac{k+1}{n}\right)^{\theta} R_n}{\sqrt{R_n}} \le \overset{o}{Z'}_n(t) \le
\frac{R'_{k+1} - \left(\frac{k}{n}\right)^{\theta} R_n}{\sqrt{R_n}},
\]
so
\[
\left| \overset{o}{Z'}_n(t) - \frac{R_{[nt]} - t^{\theta} R_n}{\sqrt{R_n}}\right|
\le 
\frac{R'_{k+1}-R'_k + \frac{1}{n^{\theta}} f_{\theta}(k) R_n}{\sqrt{R_n}}
\]
\[
\le \frac{1}{\sqrt{R_n}} + \frac{\sqrt{R_n}}{n^{\theta}} \to 0 
\]
a.s. unformly on $t \in [0,1]$.

\end{proof}

Let C(0,1) be the set of all continious functions on $[0,1]$ with the uniform metric
$\rho(x,y)=\max_{t\in[0,1]}|x(t)-y(t)|$.
By the Theorem 1, we have

\begin{corollary}
Under the assumptions of Theorem 4,
$(\overset{o}{Z}_n,\overset{o}{Z'}_n)$ converges weakly in $C(0,1)$ to 2-dimensional Gaussian process $(\overset{o}{Z},\overset{o}{Z'})$ 
that can be represented as $(\overset{o}{Z}(t),\overset{o}{Z'}(t))=(Z_\theta(t) - t^{\theta} Z_\theta(1), Z_\theta'(t) - t^{\theta} Z_\theta'(1))$, $0 \leq t \leq 1$. 
Its correlation function is given by
covariance function 
\[
c_{R,R}(s,t)=c_{R',R'}(s,t)=\overset{o}{K}(s,t), \ \ c_{R,R'}(s,t)=\overset{o}{K'}(s,t), 
\]
where
\[
\overset{o}{K}(s,t)=K(s,t)-s^{\theta}K(1,t)-t^{\theta}K(s,1)+s^{\theta}t^{\theta}K(1,1),
\]
\[
\overset{o}{K'}(s,t)=K'(s,t)-s^{\theta}K'(1,t)-t^{\theta}K'(s,1)+s^{\theta}t^{\theta}K'(1,1).
\]

\end{corollary}

Now we show how to implement the goodness-of-fit test
in this case.

Let $W_n^2=\int\limits_0^1 \left(\overset{o}{Z}_n(t)\right)^2+\left(\overset{o}{Z'}_n(t)\right)^2 dt$. It is equal to 
\begin{equation}\label{wn2}
W_n^2=\frac{1}{3n}\sum_{k=1}^{n-1} \overset{o}{Z}_n\left(\frac{k}{n}\right)\left( 2\overset{o}{Z}_n\left(\frac{k}{n}\right)+\overset{o}{Z}_n\left(\frac{k+1}{n}\right) \right)
\end{equation}
\[
+\frac{1}{3n}\sum_{k=1}^{n-1}\overset{o}{Z'}_n\left(\frac{k}{n}\right)\left( 2\overset{o}{Z'}_n\left(\frac{k}{n}\right)+\overset{o}{Z'}_n\left(\frac{k+1}{n}\right) \right).
\]

Then $W_n^2$ converges weakly to $W_{\theta}^2= \int\limits_0^1 \left(\overset{o}{Z}_{\theta}(t)\right)^2 + \left(\overset{o}{Z'}_{\theta}(t)\right)^2 dt$.

So the test rejects the basic hypothesis if $W_n^2 \geq C$. The p-value of the test is $1-F_{\theta}(W_{n,obs}^2)$.
Here $F_{\theta}$ is the cumulative distribution function of $W_{\theta}^2$ and $W_{n,obs}^2$ is a concrete value of $W_{n}^2$
 for observations under consideration.

One can estimate $F_{\theta}$ by simulations or find it explicitely using the  Smirnov's formula (Smirnov, 1937): 
if $W_{\theta}^2=\sum_{k=1}^{\infty}\frac{\eta^2_k}{\lambda_k}$,  
$\eta_1,\eta_2,\ldots$ are independent and have standard normal distribution, $0<\lambda_1<\lambda_2<\ldots$, then 
\begin{equation}\label{cdf}
F_{\theta}(x)=1+\frac{1}{\pi} \sum_{k=1}^{\infty} (-1)^k \int_{\lambda_{2k-1}}^{\lambda_{2k}}\frac{e^{-\lambda x/2}}{\sqrt{-D(\lambda)}}
\cdot \frac{d\lambda}{\lambda}, \ x>0,
\end{equation}
\[
D(\lambda)=\prod_{k=1}^{\infty} \left( 1- \frac{\lambda}{\lambda_k}\right).
\]

The integrals in the RHS of (\ref{cdf}) must tend to 0 monotonically as $k \to \infty$, 
and $\lambda_k^{-1}$ are the eigenvalues of kernel 
(see Martynov (1973), Chapter 3).

\section{Test for an  unknown rate}

Let us introduce the process $(\widehat{Z}_n,\widehat{Z'}_n)$:
\[
\widehat{Z}_n(k/n)=\left(R_k-(k/n)^{\widehat{\theta}}R_n\right)/\sqrt{R_n}, \ \ 
\widehat{Z'}_n(k/n)=\left(R'_k-(k/n)^{\widehat{\theta}}R_n\right)/\sqrt{R_n},
\]
$0\leq k \leq n$.
As for $\overset{o}{Z}_n$, let
for $0\leq u < 1/n$ and $0\leq k \leq n-1$
\[
\widehat{Z}_n\left(\frac{k}{n}+u\right)=\widehat{Z}_n(k/n)+nu\left(\widehat{Z}_n((k+1)/n) - \widehat{Z}_n(k/n)\right),
\] 
\[
\widehat{Z'}_n\left(\frac{k}{n}+u\right)=\widehat{Z'}_n(k/n)+nu\left(\widehat{Z'}_n((k+1)/n) - \widehat{Z'}_n(k/n)\right).
\]

\begin{theorem}
Under assumptions of Theorem 3, 
$(\widehat{Z}_n,\widehat{Z'}_n)$ converges weakly as $n \to \infty$ to 2-dimensional Gaussian process $(\widehat{Z}_{\theta},\widehat{Z'}_{\theta})$ 
that can be represented as $(\widehat{Z}_{\theta}(t),\widehat{Z'}_{\theta}(t))$, $0 \leq t \leq 1$, where
\[
\widehat{Z}_{\theta}(t) =\overset{o}{Z}(t) - \frac{t^{\theta}\log  t}2 
\int_0^1 u^{-\theta} (Z_{\theta}(u)+Z_{\theta}'(u))\,dA(u),
\]
\[
\widehat{Z'}_{\theta}(t) =\overset{o}{Z'}(t) - \frac{t^{\theta}\log  t}2 
\int_0^1 u^{-\theta} (Z_{\theta}(u)+Z_{\theta}'(u))\,dA(u).
\]
\end{theorem}

\begin{proof}
Similarly to the proof of Theorem 4 from the Ch.K(2019), we can show that 
\[
\sup_{t \in [0,1]} 
\left|\widehat{Z}_n(t)-(\overset{o}{Z}_n(t)- \sqrt{R_n}(\widehat{\theta}-\theta)t^{\theta}\log  t)\right| \to_{p} 0,
\]
\[
\sup_{t \in [0,1]} 
\left|\widehat{Z'}_n(t)-(\overset{o}{Z'}_n(t)- \sqrt{R_n}(\widehat{\theta}-\theta)t^{\theta}\log  t)\right| \to_{p} 0.
\]
Let's do it for $\widehat{Z'}_n(t)$. Let $t\in [0,1)$, $k=[nt]$, $u=t-k/n$, $f_{\theta}(x)=(1+x)^{\theta}-x^{\theta}$ as in the proof of Theorem 4. By the definition,
\[
\widehat{Z'}_n(k/n)=\overset{o}{Z'}_n(k/n)+\sqrt{R_n}\left((k/n)^{\theta}-(k/n)^{\widehat{\theta}}\right),
\] 
\[
\widehat{Z'}_n(t)=\overset{o}{Z'}_n(t)+\sqrt{R_n}\left((k/n)^{\theta}-(k/n)^{\widehat{\theta}}\right)
\]
\[
+nu \sqrt{R_n} \left(\left(\frac{k+1}{n}\right)^{\theta}
- \left(\frac{k+1}{n}\right)^{\widehat{\theta}}-\left(\frac{k}{n}\right)^{\theta}+\left(\frac{k}{n}\right)^{\widehat{\theta}}
\right).
\] 

We have 
\[
\left(\frac{k+1}{n}\right)^{\theta}-\left(\frac{k}{n}\right)^{\theta}=f_{\theta}(k)/n^{\theta}, \ \ 
\left(\frac{k+1}{n}\right)^{\widehat{\theta}}-\left(\frac{k}{n}\right)^{\widehat{\theta}}
=f_{\widehat{\theta}}(k)/n^{\widehat{\theta}},
\]
so
\[
\left|\widehat{Z'}_n(t)-\overset{o}{Z'}_n(t)+\sqrt{R_n}\left(t^{\widehat{\theta}}-t^{{\theta}}\right)\right|
\]
\[
=\left|\widehat{Z'}_n(t)-\overset{o}{Z'}_n(t)+\sqrt{R_n}\left(\left(\frac{k}{n}+u\right)^{\widehat{\theta}}
-\left(\frac{k}{n}+u\right)^{{\theta}}\right)\right|
\]
\[
\le 2\sqrt{R_n}\left(f_{\theta}(k)/n^{\theta}+f_{\widehat{\theta}}(k)/n^{\widehat{\theta}}\right)
\le
2\sqrt{R_n}\left(1/n^{\theta}+1/n^{\widehat{\theta}}\right)\to 0
\]
a.s. unformly on $t \in [0,1]$. 

Note that one can change $t^{\widehat{\theta}}-t^{\theta}$ by $(\widehat{\theta}-\theta)t^{\theta}\log  t $.
Really,
\[
t^{\widehat{\theta}}-t^{\theta} = t^{\theta} \left( e^{(\widehat{\theta}-\theta)\log  t} - 1\right)
\]
\[
=(\widehat{\theta}-\theta)t^{\theta}\log  t + t^{\theta}\sum_{k \ge 2} \frac{((\widehat{\theta}-\theta)\log  t)^k}{k!}
\]
\[
=(\widehat{\theta}-\theta)t^{\theta}\log  t + t^{\theta}(\widehat{\theta}-\theta)^2 (1+o(1))
\sum_{k \ge 2} \frac{\log^k  t}{k!}
\]
\[
=(\widehat{\theta}-\theta)t^{\theta}\log  t \left( 1+ (\widehat{\theta}-\theta)(1+o(1))\frac{e^{\log t} - 1 - \log t}{\log t} \right)
\]
\[
=(\widehat{\theta}-\theta)t^{\theta}\log  t ( 1+ o(1))
\]
a.s. unformly on $t \in [0,1]$. 
Hence from Theorems 3 and 4, we have joint weak convergence  of 
\[
(\overset{o}{Z}_n,\overset{o}{Z'}_n, \sqrt{R_n}(\widehat{\theta}-\theta))
\]
to 
\[
(\overset{o}{Z}_{\theta},\overset{o}{Z'}_\theta, \ \frac12 \int_0^1  u^{-\theta}( Z_{\theta}(u)+ Z'_{\theta}(u)) \, dA(u)).
\]
 So, $(\widehat{Z}_n,\widehat{Z'}_n)$ converges weakly to $(\widehat{Z}_{\theta},\widehat{Z'}_{\theta})$.

\end{proof}

\begin{corollary}  Assume the conditions of Theorem 2 to hold. 
Let $\widehat{W}_n^2=\int\limits_0^1 \left(\widehat{Z}_n(t)\right)^2+\left(\widehat{Z'}_n(t)\right)^2 dt$.
Then $\widehat{W}_n^2$ converges weakly to $\widehat{W}_{\theta}^2= \int\limits_0^1 \left(\widehat{Z}_{\theta}(t)\right)^2 +\left(\widehat{Z'}_{\theta}(t)\right)^2dt$.
\end{corollary}

Similarly to (\ref{wn2}), $\widehat{W}_n^2$ has the following representation 
\[
\widehat{W}_n^2=\frac{1}{3n}\sum_{k=1}^{n-1} \widehat{Z}_n\left(\frac{k}{n}\right)\left( 2\widehat{Z}_n\left(\frac{k}{n}\right)+
\widehat{Z}_n\left(\frac{k+1}{n}\right) \right)
\]
\[
+\frac{1}{3n}\sum_{k=1}^{n-1} \widehat{Z'}_n\left(\frac{k}{n}\right)\left( 2\widehat{Z'}_n\left(\frac{k}{n}\right)+
\widehat{Z'}_n\left(\frac{k+1}{n}\right) \right)
\]

The p-value of the goodness-of fit test is $1-\widehat{F}_{\theta}(\widehat{W}_{n,obs}^2)$.
Here $\widehat{F}_{\theta}$ is the cumulative distribution function of $\widehat{W}_{\theta}^2$, and $\widehat{W}_{n,obs}^2$ 
is the observed value of $\widehat{W}_n^2$. Further, the function
$\widehat{F}_{\theta}$ can be found using the approach from Section 3, with replacing 
$\lambda_k$ by $\widehat{\lambda}_k$ 
in the Smirnov's formula, and
$\widehat{\lambda}_k$ are the eigenvalues of the kernel.

\section{Proof of Theorem~1}

We denote by $ {\mathbb X}_i (n) $ a number of balls in urn $ i $.
Let $ \Pi=\{\Pi(t), t \geq 0\}$ be a Poisson process with parameter $1$. The Poissonized version of Karlin model assumes the total number of $\Pi(n)$ balls. According to well-known thinning property of Poisson flows, stochastic processes $\left\{{\mathbb X}_{i}(\Pi(t)) \stackrel{\text { def }}{=} \mathbf{\Pi}_{i}(t), t \geq 0\right\}$ are compound Poisson with intensities $p_{i}$ and are mutually independent for different $i$ 's. The definition implies that for any fixed $n\ge1$, $\tau,t\in[0,1]$
\[
R_{\Pi(t n)}=\sum_{k=1}^{\infty} \mathbf{I}\left(\mathbf{\Pi}_{k}(t n) >0\right)=\sum_{k=1}^{\infty} \mathbf{I}_k(tn), 
\]
\[
R'_{\Pi(\tau n)}=\sum_{k=1}^{\infty} \mathbf{I}\left(\mathbf{\Pi}_{k}(n)-\mathbf{\Pi}_{k}((1-\tau) n)>0\right)=\sum_{k=1}^{\infty} \mathbf{I}'_k(\tau n) .
\]

{\bf Step 1 (covariances)} Let $\tau, t\in[0,1]$ 
\[
{\bf cov}\left(R_{\Pi(t n)}, R'_{\Pi(\tau n)}\right) =\sum_{k=1}^{\infty}{\bf cov}( \mathbf{I}_k(tn),\mathbf{I}'_k(\tau n))
\]
\[
=\sum_{k=1}^{\infty}\left(\mathbf{P}\left(\mathbf{\Pi}_{k}(t n)>0, \mathbf{\Pi}_{k}(n)-\mathbf{\Pi}_{k}((1-\tau)n)>0\right)-(1-e^{-p_k t n})(1-e^{-p_k \tau n})\right)
\]
Note that if $t+\tau>1$, then
\[
\mathbf{P}(\mathbf{\Pi}_{k}(t n)>0, \mathbf{\Pi}_{k}(n)-\mathbf{\Pi}_{k}((1-\tau)n)>0)=
\mathbf{P}(\mathbf{\Pi}_{k}(t n)-\mathbf{\Pi}_{k}((1-\tau)n)>0)
\]
\[
+\mathbf{P}(\mathbf{\Pi}_{k}(t n)-\mathbf{\Pi}_{k}((1-\tau)n)=0,\mathbf{\Pi}_{k}((1-\tau)n)>0,\mathbf{\Pi}_{k}(n)-\mathbf{\Pi}_{k}(t n)>0)
\]
\[
=1-e^{-p_k(t+\tau-1)n}+e^{-p_k(t+\tau-1)n}(1-e^{-p_k(1-\tau)n})(1-e^{-p_k(1-t)n})
\]
\[
=1-e^{-p_k t n}-e^{-p_k \tau n}+e^{-p_k n}.
\]
Hence
\[
{\bf cov}\left(R_{\Pi(t n)}, R'_{\Pi(\tau n)}\right) =\mathbf{I}(t+\tau>1)\sum_{k=1}^{\infty}\left( e^{-p_k(t+\tau)n}-e^{-p_k n}\right)
\]
\[
=\mathbf{I}(t+\tau>1) (\E R_{\Pi((t+\tau)n}-\E R_{\Pi(n)}).
\]
Since 
 \[
{\bf E} R_{\Pi (t n)}/{\bf E} R_{n}\sim t^\theta,  
\]
then
\[
{\bf cov}\left(R_{\Pi(t n)}, R'_{\Pi(\tau n)}\right)/ {\bf E} R_{n}\sim K'(t,\tau).
\]

{\bf  Step~2 (convergence of finite-dimensional distributions)}
Analogously to proof of Theorem 1 in Dutko (1989) we have that, for any fixed $m \geq 1,0<t_{1}<t_{2}<\cdots<t_{m} \leq 1$ triangle array of $2 m $ -dimensional random vectors 
$\left\{
( ( \mathbf{I}_k(n t_{i})-\mathbf{E}\mathbf{I}_k(n t_{i}) )/  \sqrt{{\bf E} R_{n}}, ( \mathbf{I}'_k(n t_{i})-\mathbf{E}\mathbf{I}'_k(n t_{i}) )/  \sqrt{{\bf E} R_{n}}, i \leq m), k \leq n
\right\}_{n \geq 1}$ 
satisfies Lindeberg condition (see Borovkov (2013), Theorem $8.6 .2,$ p. 215 ).

{\bf Step~3 (relative compactness)} 
Since $R_{\Pi(nt)}\stackrel{d}{=}R'_{\Pi(nt)}$, then relative compactness follows from Chebunin and Kovalevskii (2016).

{\bf Step 4 (approximation of the original process) }

It follows from the corresponding step of proof in Chebunin and Kovalevskii (2016) and the previous step.

{\it Theorem~1 is proved.}

\bigskip

{\bf Acknowledgements}

The work is supported by Mathematical Center in Akademgorodok under agreement No. 075-15-2019-1675 with the Ministry of Science and Higher Education of the Russian Federation.

\end{document}